\documentclass[a4paper,11pt,reqno]{amsart}
\usepackage[T1]{fontenc}
\usepackage[utf8]{inputenc}
\usepackage[margin=1.0in]{geometry}
\usepackage{lmodern}
\usepackage[english]{babel}
\usepackage{graphicx}
\usepackage{amsmath}
\usepackage{amsfonts}
\usepackage{amssymb} 
\usepackage{amsthm}
\usepackage{bbm}
\usepackage{bm} 
\usepackage{lipsum}
\usepackage{xcolor}
\usepackage{url}
\usepackage{hyperref}
\usepackage{stmaryrd}

\allowdisplaybreaks

\theoremstyle{plain}
\newtheorem{thm}{Theorem}%[section]
\newtheorem{lem}[thm]{Lemma}%[section]
\newtheorem{cor}[thm]{Corollary}
\newtheorem{rem}[thm]{Remark}%[section]
\theoremstyle{definition}
\newtheorem{innercustomgeneric}{\customgenericname}
\providecommand{\customgenericname}{}
\newcommand{\newcustomtheorem}[2]{%
  \newenvironment{#1}[1]
  {%
   \renewcommand\customgenericname{#2}%
   \renewcommand\theinnercustomgeneric{##1}%
   \innercustomgeneric
  }
  {\endinnercustomgeneric}
}

\newcustomtheorem{customthm}{Theorem}

\def \be {\begin{equation}}
\def \ee {\end{equation}}

\def \R {\mathbb{R}}

\newcommand{\eps}{\varepsilon}
\renewcommand{\phi}{\varphi}
\renewcommand{\epsilon}{\varepsilon}

\renewcommand{\hat}{\widehat}
\renewcommand{\bar}{\overline}

\begin{document}

\title[Convergence rate for vanishing viscosity of fully nonlinear HJ equations]{Convergence rates for the vanishing viscosity approximation of fully nonlinear, non-convex, second-order Hamilton-Jacobi equations}

\author{Alekos Cecchin}
\address{Dipartimento di Matematica ``Tullio Levi-Civita'', Universit\`a degli Studi di Padova, 
via Trieste 63, 35121 Padova (Italy)}
\curraddr{}
\email{alekos.cecchin@unipd.it}
\thanks{}
\author{Alessandro Goffi}
\address{Dipartimento di Matematica e Informatica ``Ulisse Dini'', Universit\`a degli Studi di Firenze, 
viale G. Morgagni 67/a, 50134 Firenze (Italy)}
\curraddr{}
\email{alessandro.goffi@unifi.it}

\date{\today}
\thanks{A.C. acknowledges support from the project MeCoGa “Mean field control and games” of the University of Padova
through the program STARS@UNIPD, the PRIN 2022 Project 2022BEMMLZ
“Stochastic control and games and the role of information”, the INdAM-GNAMPA Project 2025 “Stochastic control
and MFG under asymmetric information: methods and applications” and the PRIN 2022 PNRR Project P20224TM7Z
“Probabilistic methods for energy transition”. A.G. is supported by the INdAM-GNAMPA Project 2025 ``Stabilit\`a e confronto per equazioni di trasporto/diffusione e applicazioni a EDP non lineari''}

\begin{abstract}
We obtain new quantitative estimates of the vanishing viscosity approximation for time-dependent, degenerate, Hamilton-Jacobi equations that are neither concave nor convex in the gradient and Hessian entries of the form $\partial_t u+H(x,t,Du,D^2u)=0$ in the whole space. We approximate the PDE with a fully nonlinear, possibly degenerate, elliptic operator $\eps F(x,t,D^2u)$. Assuming that $u\in C^\alpha_x$, $u_0\in C^\eta$, $H\in C^\beta_x$ and having power growth $\gamma$ in the gradient entry, we establish a convergence rate of order $\eps^{\min\left\{\frac{\eta}{2},\frac{\beta+\gamma(\alpha-1)}{\beta+\gamma(\alpha-1)+2-\alpha}\right\}}$. Our novel approach exploits the regularizing properties of sup/inf-convolutions for viscosity solutions and the comparison principle. We also obtain explicit constants and do not assume differentiability properties neither on solutions nor on $H$. The same method provides new convergence rates for the vanishing viscosity approximation of the stationary counterpart of the equation and for transport equations with H\"older coefficients. 
\end{abstract}

\maketitle

\section{Introduction}
Aim of this note is to investigate the rate of convergence in sup-norm of the (nonlinear) vanishing viscosity approximation of degenerate Hamilton-Jacobi-Isaacs equations
\begin{equation}\label{HJintro}\tag{HJ}
\begin{cases}
\partial_t u(x,t)+H(x,t,Du(x,t),D^2u(x,t))=0&\text{ in }\R^n\times(0,T),\\
u(x,0)=u_0(x)&\text{ in }\R^n.
\end{cases}
\end{equation}
under minimal regularity assumption in the spatial variable of $H$, the initial datum $u_0$ and the solution $u$. We consider the approximation
\begin{equation}\label{HJB}\tag{$\mathrm{HJ_\varepsilon}$}
\begin{cases}
\partial_t u_{(\varepsilon)}+H(x,t,u_{(\varepsilon)},Du_{(\varepsilon)},D^2 u_{(\varepsilon)})-\varepsilon F(x,t,D^2u_{(\varepsilon)})=0&\text{ in }\R^n\times(0,T),\\
u_{(\varepsilon)}(x,0)=u_0(x)&\text{ in }\R^n.
\end{cases}
\end{equation}

The main peculiarities of our results are that $H$ is neither concave nor convex in the gradient variable and it can be fully nonlinear, degenerate, in the matrix entry. In addition, our viscosity operator $F$ needs not be uniformly elliptic: we will need only its monotonicity and a certain Lipschitz condition in the matrix variable. Under these assumptions $u_{(\eps)}$ might not be smooth, but merely (H\"older) continuous. To our knowledge these convergence rates have never been addressed before, especially in the context of  fully nonlinear second-order equations. Moreover, they are new even when $F(x,t,D^2u_{(\eps)})=\Delta u_{(\eps)}$ is linear and $H$ is of first-order. The main novelty relies on the method of proof: it is based on the regularization of viscosity solutions by sup- and inf-convolutions \cite{LasryLions} and the comparison principle \cite{CIL}. We avoid both the doubling of variable technique and representation formulas for $u$ as the value of a stochastic differential game. Since both $H$ and $F$ are fully nonlinear, second-order, non-differentiable operators, integral methods cannot be applied; moreover, other methods based on the differentiation of the equation are not allowed, as neither $H$ nor $u$ are assumed to be differentiable. Furthermore, we pay attention to the quantitative dependence of the constants with respect to the time variable and the dimension of the ambient space, and provide sharp explicit constants depending on the H\"older seminorms of $u$ and $u_0$, as well as the H\"older regularity of $H$.

Let us summarize here the main results on the subject for the classical elliptic regularization $\varepsilon\Delta u$ and $H=\widetilde H(x,t,Du)$, i.e. when $u_{(\varepsilon)}$ solves
\[
\partial_t u_{(\varepsilon)}+\widetilde H(x,t,Du_{(\varepsilon)})=\varepsilon\Delta u_{(\varepsilon)}.
\]
\begin{itemize}
\item When $\widetilde H(p)\in W^{1,\infty}_{\mathrm{loc}}(\R^n)$ and $u_{(\varepsilon)}\in W^{1,\infty}_x$, classical results obtained by P.-L. Lions \cite{L82book}, M.-G. Crandall-P.-L. Lions \cite{CL84} and L.C. Evans \cite{EvansArma} show that
\[
\|(u_{(\varepsilon)}-u)(t)\|_{L^\infty(\R^n)}\leq C\sqrt{t\varepsilon}.
\] 
There are several proofs of this result: it can be obtained via stochastic control methods, doubling of variable techniques, or integral duality methods respectively. The same rate holds for a general $\widetilde H(x,t,Du)$ under Lipschitz-type continuity assumptions on the space variable \cite{Souganidis}. In the case of Lipschitz data this is the best one can expect from the heat flow, as shown in Lemma 1.14 of \cite{BertozziMajda} for the heat equation, see also \cite{Sprekeleretal};
\item When $u$ is semisuperharmonic, i.e. $\Delta u_{(\varepsilon)}(t)\leq c(t)\in L^1_t$, we have a one-side linear rate with respect to the viscosity \cite{L82book,CGM};
\item Geometric conditions, such as the convexity of the data, or smallness assumptions (e.g. short-time horizons), guarantee a two-side linear rates of convergence $\mathcal{O}(\varepsilon)$, cf. \cite{CG25};
\item The rate can be improved to $\mathcal{O}(\varepsilon)$ in $L^1$ norms \cite{LinTadmor,CGM}, and this continues to be valid also for some Hamiltonians that are neither concave nor convex. An interpolation gives $\mathcal{O}(\varepsilon^{\nu(p)})$ rates in $L^p$ norms, with $\nu(p)\in(\frac12,1)$ and $\nu(p)\to\frac12$ as $p\to+\infty$;
\item When $H$ is uniformly convex the rate is 
\[
\|u_{(\varepsilon)}-u\|_{L^\infty(\R^n\times(0,T))}\leq C\varepsilon|\log\varepsilon|,
\]
see \cite{CG25}. This is proved using duality techniques, and it is optimal, cf. \cite{Brenier,Sprekeleretal,ChaintronDaudin}. Under these geometric conditions on $H$, one can reach $\mathcal{O}(\eps)$ in average; see \cite{TranBook}.
\item When $H$ is strictly convex and exhibits superquadratic behavior, i.e. $\widetilde H(p)\sim |p|^\gamma$, $\gamma>2$, we have
\[
\|(u_{(\varepsilon)}-u)(t)\|_{L^\infty(\R^n)}\leq C\varepsilon^{\mu(\gamma)}
\]
for some $\mu(\gamma)\in(\frac12,1)$. The proof combines a higher-order version of the hole-filling method by Widman and the Evans' adjoint technique; see \cite{CG25}. \end{itemize}
Little is known under merely continuity or H\"older continuity assumptions on $H$, except the rate of convergence for the vanishing viscosity approximation of the stationary problem
\[
u+\bar H(x,Du)=0\text{ in }\R^n.
\]
The authors in \cite[Theorem 3.2 and Remark 3.3]{BCD} proved by doubling of variables techniques the rate
\[
\|(u_\varepsilon-u)(t)\|_{L^\infty(\R^n)}\leq C\varepsilon^{\frac{\alpha}{2}},
\]
when $H$ is neither concave nor convex in $Du$, satisfies 
\begin{equation*}
|\bar H(x,p)-\bar H(y,p)|\leq L|x-y|(1+|p|)
\end{equation*}
and $u\in C^{\alpha}(\R^n)$. Related results appeared in \cite{CCM} for the vanishing viscosity of fully nonlinear second-order homogenization problems. These estimates are based on the interplay between the vanishing viscosity approximation and the effect of the homogenization, and thus do not address the usual vanishing viscosity process.\\

Here we assume that $H$ satisfies the inequality
\begin{equation}\label{assH}\tag{H}
|H(x,t,p,M)-H(y,t,p,M)|\leq C_H|x-y|^\beta(1+|p|^\gamma)
\end{equation}
whenever $(x,t),(y,t)\in\R^n\times(0,T)$, $p\in\R^n$, $M\in\mathrm{Sym}_n$, for
\begin{equation}
	\label{eq:1}
\beta\in(0,1]\text{ and }\gamma\geq0
\end{equation}
with $u\in C((0,T);C^\alpha(\R^n))$, $\alpha\in(0,1]$, and $\alpha,\beta$ and $\gamma$ satisfy
\begin{equation}\label{compatibility}
\beta+(\alpha-1)\gamma>0.
\end{equation}
We also suppose $u_0\in C^\eta(\R^n)$, with $\eta\in(0,1]$, $\eta\leq\alpha$. As far as $F$ is concerned, we assume it satisfies
\begin{equation}\label{ell}
F(x,t,M+N)-F(x,t,M)\leq \Lambda\mathrm{Tr}(N),\ \quad \forall M,N\in\mathrm{Sym}_n,\ N\geq0,
\end{equation}
and that it is monotone increasing in the last entry, i.e.
\begin{equation}\label{mon}
F(x,M_1)\leq F(x,M_2)\quad\text{ if }M_1\leq M_2,\ M_1,M_2\in\mathrm{Sym}_n.
\end{equation}
with
\begin{equation}\label{Cf}
|F(x,t,0)|\leq C_F.
\end{equation}
Note that $F$ is not assumed to be uniformly elliptic, but the standard uniform ellipticity implies both \eqref{ell} and \eqref{mon}, cf. \cite[Lemma 2.2]{CC}
\footnote{In fact, \eqref{ell} and \eqref{mon} are equivalent to say that 
	
	$0\leq F(x,t,M+N)-F(x,t,M)\leq \Lambda\mathrm{Tr}(N)$ for any $M,N\in\mathrm{Sym}_n$,  $N\geq0$.
}
. 
In the linear case $F(x,t,M)=\mathrm{Tr}(A(x,t)M)$, our hypotheses are satisfied if $0\leq A(x,t)\leq \Lambda \mathbb{I}_n$.

We emphasize that $H$ can have any growth in the gradient entry. The classes of operators $H$ and $F$ encompass Hamilton-Jacobi-Bellman and Isaacs operators arising in stochastic control and differential games theory, but also the case of transport equations with H\"older coefficients \cite{LionsSeeger}. We also assume that \eqref{HJintro} and its viscous approximation \eqref{HJB} satisfy the comparison principle \cite{CIL}:
\begin{equation}\label{CP}\tag{$\mathrm{CP_\varepsilon}$}
\begin{split}
\text{$u$ viscosity subsol., $v$ viscosity supersol. }& \text{of \eqref{HJB}, uniformly continuous in $\R^n\!\times\! [0,T]$, }\\
 \text{with } u(\cdot,0)&\leq v(\cdot,0)\text{ in }\R^n
\\
&\Downarrow\\
u(x,t)\leq v(x,t)&\text{ in }\R^n\!\times\!(0,T).
\end{split}
\end{equation}
Note that this assumption hides certain structural growth conditions of $H$ in the Hessian and the gradient and of the solution of our problem.

Our result highlights that one can approximate the PDE with any general fully nonlinear elliptic operator (e.g. in Isaacs form), or even with a linear operator in nondivergence form with continuous coefficients. Therefore the viscosity approximation does not depend on the structure of the operator. We focus on the interplay between the H\"older regularity assumptions on $u$ and $H$. We also consider different exponents $\eta$ of $u_0$ and $\alpha$ of $u$ to take into account possible smoothing effects for positive times independent of the viscosity $\eps$.  Finally, our method provides new results on the vanishing viscosity approximation of H\"older solutions to transport-equations with H\"older continuous coefficients. They agree with recent findings by P.-L. Lions- B. Seeger \cite{LionsSeeger}, producing a quantitative selection principle for such transport problems.
\\

\paragraph{\textbf{Main contributions}} We summarize our results on some model problems:
\begin{itemize} 
\item Let $u_{(\varepsilon)}$ be the solution of  \eqref{HJB} with $F$ and $H$ as above, $u_0\in C^1(\R^n)$, then 
\[
\|(u_{(\varepsilon)}-u)(t)\|_{L^\infty(\R^n)}\leq C(t)\varepsilon^{\frac{\beta+\gamma(\alpha-1)}{\beta+\gamma(\alpha-1)+2-\alpha}}.
\]
Here and below the constants of the estimates $C(t)$ go to 0 as $t\to0^+$. In the special case $\gamma=0$ we deduce
\[
\|(u_{(\varepsilon)}-u)(t)\|_{L^\infty(\R^n)}\leq C(t)\varepsilon^{\frac{\beta}{\beta+2-\alpha}}.
\]
If $\beta=\alpha=1$ and $F(x,D^2u_{(\varepsilon)})=\varepsilon\Delta u_{(\varepsilon)}$ is linear, we recover the classical $\mathcal{O}(\sqrt{\varepsilon})$ rate for Lipschitz solutions and locally Lipschitz $H$ in $Du$. If $\beta=\alpha\in(0,1)$ we get the estimate
\[
\|(u_{(\varepsilon)}-u)(t)\|_{L^\infty(\R^n)}\leq C(t)\varepsilon^\frac{\alpha}{2}.
\]
This holds under weaker regularity assumptions than \cite{BCD}. The same rate is obtained in the regime $\gamma=\beta=1$, which is the one considered in \cite[Theorem VI.3.2]{BCD} for the stationary problem. If $\gamma=1$ and $F$ is linear (the case of transport-diffusion equations, cf. \cite{LionsSeeger}) the approach does not work unless \eqref{compatibility} holds, i.e.
\[
\beta+\alpha>1.
\] 
This agrees with \cite[Theorem 3.7 and Remark 3.1]{LionsSeeger}, where such condition is optimal and arises in connection with uniqueness results of linear equations with one-side Lipschitz coefficients and H\"older continuous coefficients.
\item Our more general result is Theorem \ref{main}, where all the constants are explicit. If $u_0\in C^{\eta}(\R^n)$ we have the estimate
\[
\|(u_{(\varepsilon)}-u)(t)\|_{L^\infty(\R^n)}\leq C(t)\varepsilon^{\min\left\{\frac{\beta+\gamma(\alpha-1)}{\beta+\gamma(\alpha-1)+2-\alpha},\frac{\eta}{2}\right\}}.
\]
In the special case $\gamma=0$, it reduces to
\[
\|(u_{(\varepsilon)}-u)(t)\|_{L^\infty(\R^n)}\leq C(t)\varepsilon^{\min\left\{\frac{\beta}{\beta+2-\alpha},\frac{\eta}{2}\right\}}.
\]
When $\alpha=1$ and $\gamma=0$ we have
\[
\|(u_{(\varepsilon)}-u)(t)\|_{L^\infty(\R^n)}\leq C(t)\varepsilon^{\min\left\{\frac{\beta}{\beta+1},\frac{\eta}{2}\right\}}.
\]

\item For stationary problems of the form
\[
\rho u+H(x,Du,D^2u)=0\text{ in }\R^n,\ \rho>0,
\] 
the same method provides the convergence rate 
\[
\|u_{(\varepsilon)}-u\|_{L^\infty(\R^n)}\leq C\varepsilon^{\frac{\beta+\gamma(\alpha-1)}{\beta+\gamma(\alpha-1)+2-\alpha}}.
\]
\end{itemize}

\textit{Outline}. Section \ref{sec;conv} recalls some properties of sup-/inf- convolutions. Section \ref{sec;main} provides the main result Theorem \ref{main} and some consequences such as its counterpart in the stationary case Theorem \ref{mainstat} and a convergence result towards the initial datum for fully nonlinear evolution equations stated in Corollary \ref{mainfully}.

\section{Properties of sup- and inf-convolutions}\label{sec;conv}
We denote 
\[
u^\delta(x,t)=\sup_{y\in\R^n}\left\{u(y,t)-\frac{1}{2\delta}|x-y|^2\right\},
\]
\[
u_\delta(x,t)=\inf_{y\in\R^n}\left\{u(y,t)+\frac{1}{2\delta}|x-y|^2\right\},
\]
cf. \cite{LasryLions,CC}. We denote by $u(t)=u(\cdot,t)$ and $[u(t)]_\alpha$ the H\"older seminorm in the space variable of the function $u(t)$ for $\alpha\in(0,1]$, while $[u(t)]_0$ is the oscillation. 
\begin{lem}[Speed of convergence of nonlinear convolutions] If $u(t)\in C^\alpha(\R^n)$, $\alpha\in[0,1]$, we have for all $t\in[0,T]$
\begin{equation}\label{speeddist1}
|x-x^\delta|^{2-\alpha}\leq 2\delta[u(t)]_\alpha\text{ where }x^\delta\in\mathrm{argmax}_{y\in \R^n}\left\{u(y,t)-\frac{1}{2\delta}|x-y|^2\right\},
\end{equation}
\begin{equation}\label{speeddist2}
|x-x_\delta|^{2-\alpha}\leq 2\delta[u(t)]_\alpha\text{ where }x_\delta\in\mathrm{argmin}_{y\in \R^n}\left\{u(y,t)+\frac{1}{2\delta}|x-y|^2\right\},
\end{equation}

\begin{equation}\label{speed1}
\|(u^\delta-u)(t)\|_{L^\infty(\R^n)}\leq (2[u(t)]_{\alpha})^\frac{2}{2-\alpha}\delta^\frac{\alpha}{2-\alpha},
\end{equation}
\begin{equation}\label{speed2}
\|(u_\delta-u)(t)\|_{L^\infty(\R^n)}\leq (2[u(t)]_{\alpha})^\frac{2}{2-\alpha}\delta^\frac{\alpha}{2-\alpha},
\end{equation}
Moreover, $u^\delta$ and $u_\delta$ are Lipschitz continuous with 
\begin{equation}\label{reggrad}
\|Du^\delta(t)\|_{L^\infty(\R^n)}\leq \delta^{-\frac{1-\alpha}{2-\alpha}}(2[u(t)]_{\alpha})^{\frac{1}{2-\alpha}},\quad \|Du_\delta(t)\|_{L^\infty(\R^n)}\leq \delta^{-\frac{1-\alpha}{2-\alpha}}(2[u(t)]_{\alpha})^{\frac{1}{2-\alpha}}.
\end{equation}
\end{lem}
\begin{proof}
For \eqref{speeddist1}-\eqref{speed2} see \cite[Lemma 8.5-(iii)]{Calder}. The proof of \eqref{reggrad} can be found in \cite[Proposition 4.1]{Fujita}, recalling that for fixed $s\in(0,T)$ the solution of 
\[
\begin{cases}
\partial_t v+\frac{|Dv|^2}{2}=0&\text{ in }\R^n\times(0,T),\\
v(x,0)=u(x,s)&\text{ in }\R^n
\end{cases}
\]
is given by the Hopf-Lax formula
\[
v(x,t)=\inf_{y\in\R^n}\left\{ u(y,s)-\frac{|x-y|^2}{2t}\right\}.
\]
Consequently, for $t=\delta$ one can write
\[
u_{\delta}(x,s)=v(x,\delta)
\]
and derive \eqref{reggrad} from the regularizing effects of solutions to first-order Hamilton-Jacobi equations. Similar properties hold for the sup-convolution $u^\delta$, recalling that $u_\delta=-(-u)^\delta$.
\end{proof}
In general, the sup/inf-convolutions are not $C^1$, but have important semiconvexity/semiconcavity properties:
\begin{lem}[Geometric and monotonicity properties of nonlinear convolutions]\label{prop2} Let $u:\R^n\times(0,T)\to\R$ be continuous. We have for all $t\in(0,T)$
\begin{itemize}
\item[(i)] $u_\delta(t)\leq u(t)\leq u^\delta(t)$;
\item[(ii)] $u^\delta(t)+\frac{1}{2\delta}|x|^2$ is convex (i.e. $u^\delta(t)$ is semiconvex with constant $\frac1\delta$) and $u_\delta(t)-\frac{1}{2\delta}|x|^2$ is concave (i.e. $u_\delta(t)$ is semiconcave with constant $\frac1\delta$).
\end{itemize}
\end{lem}
\begin{proof}
See \cite[Lemma II.4.11]{BCD}, \cite[Chapter 5]{CC} or \cite[Proposition 8.3]{Calder}.
\end{proof}

\section{Main results}\label{sec;main}
\begin{thm}\label{main}
 Let $\epsilon>0$ and assume that \eqref{assH}, \eqref{eq:1}, \eqref{ell},\eqref{mon}, \eqref{Cf} and \eqref{CP} hold. Let $u_{(\varepsilon)}$ and $u$  be  viscosity solutions to \eqref{HJB} and \eqref{HJintro} respectively, uniformly continuous on $ \R^n\times[0,T]$. Let $\alpha\in(0,1]$ and assume that $u\in C((0,T];C^\alpha(\R^n)) \cap L^{\frac{\beta+\gamma}{2-\alpha}}((0,T);C^\alpha(\R^n))$, 
and that \eqref{compatibility} is satisfied. Suppose also that $u_0\in C^\eta(\R^n)$, $\eta\in(0,1]$ with $\eta\leq\alpha$. Then for all $t\in[0,T]$
\begin{equation}
\label{eq:main}
\|(u_{(\varepsilon)}-u)(t)\|_{L^\infty(\R^n)}\leq (C_1+C_2(t))^{\frac{1}{P+1}}\left(\frac{P+1}{P^{\frac{P}{P+1}}}\right)(\Lambda t n\varepsilon)^{\frac{P}{P+1}}+\eps t C_F,
\end{equation}
where 
\begin{align*}
P:&=
\begin{cases}
\min\left\{\frac{\eta}{2-\eta},\frac{\beta+\gamma(\alpha-1)}{2-\alpha}\right\}&\quad\text{ if }C_1,C_H\neq0,\\
\frac{\eta}{2-\eta}&\quad\text{ if }C_H=0,\\
\frac{\beta+\gamma(\alpha-1)}{2-\alpha}&\quad\text{ if }C_1=0,\\
\end{cases}\\
\frac{P}{P+1}&=
\begin{cases}
\min\left\{\frac{\eta}{2},\frac{\beta+\gamma(\alpha-1)}{\beta+\gamma(\alpha-1)+2-\alpha}\right\}&\text{ if }C_1,C_H\neq0,\\
\frac{\eta}{2}&\text{ if }C_H=0,\\
\frac{\beta+\gamma(\alpha-1)}{\beta+\gamma(\alpha-1)+2-\alpha}&\text{ if }C_1=0,\\
\end{cases}\\
C_1:&=(2[u_0]_{\eta})^\frac{2}{2-\eta},\qquad C_2(t):=C_H\int_0^t(2[u(s)]_\alpha)^{\frac{\beta}{2-\alpha}}\left(1+(2[u(s)]_\alpha)^{\frac{\gamma}{2-\alpha}}\right)\,ds.
\end{align*}
\end{thm}

\begin{proof}
\textit{Step 1}. We show that the sup-convolution $u^\delta$ of $u$ is a viscosity subsolution
of 
\begin{equation}\label{supsub}
\begin{cases}
\partial_t u^\delta+H(x,t,Du^\delta,D^2u^\delta)\leq R(t,\delta)&\text{ in }\R^n\times(0,T),\\
u^\delta(x,0)\leq u_0(x)+R_0(\delta)&\text{ in }\R^n,
\end{cases}
\end{equation}
where
\[
R(t,\delta)=C_H(2\delta [u(t)]_\alpha)^{\frac{\beta}{2-\alpha}}\left(1+(2\delta)^{\frac{\gamma(\alpha-1)}{2-\alpha}}[u(t)]_\alpha^{\frac{\gamma}{2-\alpha}}\right),
\]
\[
R_0(\delta)=(2[u_0]_{\eta})^\frac{2}{2-\eta}\delta^\frac{\eta}{2-\eta}.
\]
This is well-known for stationary problems under certain restrictions on $\alpha$ when $\gamma=0$. When $\alpha\in(0,1]$; see \cite[Remark 8.10]{Calder}. When $\alpha=\gamma=0$ and $F$ uniformly continuous the remainder is implicit and can be found in \cite[Lemma 3.1]{BBDL}.
Let $\varphi(x,t)$ be a $C^{2,1}(\R^n\times(0,T))$ test function, let $(\bar x,\bar t)$ be a local maximum point of $u^\delta-\varphi$ and let $\bar x^\delta$ be as in \eqref{speeddist1}. It is easy to show, cf. \cite[Proposition 8.6]{Calder}, that $(\bar x^\delta,\bar t)$ is a local maximum point of 
\[
(y,t)\longmapsto u(y,t)-\varphi(y+\bar x-\bar x^\delta,t).
\]
Hence, the viscosity subsolution property of $u$ yields
\[
\partial_t\varphi (\bar x,\bar t)+H(\bar x^\delta,\bar t,D\varphi(\bar x,\bar t),D^2\varphi(\bar x,\bar t))\leq 0.
\]
Using \eqref{assH}, \eqref{speeddist1},\eqref{reggrad} and the fact that on the maximum point $|D\varphi(\bar x,\bar t)|\leq \|Du^\delta\|_\infty$,  
we obtain
\begin{align*}
\partial_t\varphi (\bar x,\bar t)+H(\bar x,\bar t,D\varphi(\bar x,\bar t),D^2\varphi(\bar x,\bar t))&\leq C_H|\bar x-\bar x^\delta|^\beta(1+|D\varphi(\bar x,\bar t)|^\gamma)\\
& \leq C_H(2\delta[u(t)]_\alpha)^\frac{\beta}{2-\alpha}(1+\delta^{-\frac{\gamma(1-\alpha)}{2-\alpha}}(2[u(t)]_{\alpha})^{\frac{\gamma}{2-\alpha}}),
\end{align*}
which gives 
\begin{equation}
\label{eq:10}
\partial_t\varphi (\bar x,\bar t)+H(\bar x,\bar t,D\varphi(\bar x,\bar t),D^2\varphi(\bar x,\bar t))\leq R(t,\delta).
\end{equation}
Using the bound  \eqref{speed1} for the initial condition, we get that $u^\delta$ is a viscosity subsolution to \eqref{supsub}.

\textit{Step 2}. We first recall that $u^\delta$ is semiconvex with constant $\frac{1}{\delta}$; hence on the maximum point  $D^2\varphi(\bar x,\bar t)\geq-\frac{1}{\delta}$. The one-side ellipticity of $F$ \eqref{ell} applied with $M=D^2\varphi(\bar x,\bar t)$, and $N=\frac{1}{\delta}\mathbb{I}_n$, together with the monotonicity \eqref{mon} and the bound \eqref{Cf} give the inequalities
\begin{align*}
-F(\bar x,\bar t,D^2\varphi(\bar x,\bar t))&\leq -F\left(\bar x,\bar t,D^2\varphi(\bar x,\bar t)+\frac{1}{\delta}\mathbb{I}_n\right)+ \Lambda \mathrm{Tr}\left(\frac{1}{\delta}\mathbb{I}_n\right)\\
&\leq -F\left(\bar x,\bar t,0\right)+\frac{\Lambda n}{\delta}\\
&\leq C_F+\frac{\Lambda n}{\delta}.
\end{align*}
Therefore, adding $-\eps F(\bar x,\bar t,D^2\varphi(\bar x,\bar t))$ on both sides of the inequality \eqref{eq:10},
we obtain
\[
\partial_t\varphi (\bar x,\bar t)-\eps F(\bar x,\bar t,D^2\varphi(\bar x,\bar t))+H(\bar x,\bar t,D\varphi(\bar x,\bar t),D^2\varphi(\bar x,\bar t))\leq R(t,\delta)+\frac{\eps\Lambda n}{\delta}+\eps C_F.
\]
Thus, the function
\[
\hat u^{\delta,\eps}(x,t)=u^\delta(x,t)-\int_0^t R(s,\delta)\,ds-t\left(\frac{\eps\Lambda n}{\delta}+\eps C_F\right)-R_0(\delta)
\]
is a viscosity subsolution of \eqref{HJB}.

\textit{Step 3}. Applying \eqref{CP} we get $\hat u^{\delta,\eps}(x,t)\leq u_{(\varepsilon)}(x,t)$ and hence
\[
u^\delta(x,t)-u_{(\varepsilon)}(x,t)\leq \int_0^t R(s,\delta)\,ds+t\left(\frac{\eps\Lambda n}{\delta}+\eps C_F\right)+R_0(\delta).
\]
Using the monotonicity of $u^\delta$ in Lemma \ref{prop2}-(i), that is $u^\delta\geq u$, we get
\begin{align*}
u(x,t)-u_{(\varepsilon)}(x,t)&\leq t\frac{\eps\Lambda n}{\delta}+\underbrace{(2[u_0]_{\eta})^\frac{2}{2-\eta}}_{=:C_1}\delta^\frac{\eta}{2-\eta}\\
&+\delta^{\frac{\beta+\gamma(\alpha-1)}{2-\alpha}}\underbrace{C_H\int_0^t(2[u(s)]_\alpha)^{\frac{\beta}{2-\alpha}}\left(1+(2[u(s)]_\alpha)^{\frac{\gamma}{2-\alpha}}\right)\,ds}_{=:C_2(t)}+t\eps C_F\\
&\leq t\frac{\eps\Lambda n}{\delta}+t\eps C_F+(C_1+C_2(t))\delta^P,
\end{align*}
where $P$ is the exponent given in the statement. Optimizing with respect to $\delta$ we get
\[
\delta=\left(\frac{t\eps\Lambda n}{(C_1+C_2(t))\cdot P}\right)^{\frac{1}{P+1}},
\]
which gives a one-side bound in \eqref{eq:main}.

\textit{Step 4}. The reversed estimate is obtained using inf-convolutions instead of sup-convolutions and arguing with viscosity supersolutions instead of subsolutions. In Step 2, instead, we apply \eqref{ell} with $M=D^2\varphi(\bar x,\bar t)-\frac{1}{\delta}\mathbb{I}_n$ and $N=\frac{1}{\delta}\mathbb{I}_n$ and the semiconcavity property of $u_\delta$, which implies the one-side bound $D^2\varphi(\bar x,\bar t)\leq \frac{1}{\delta}\mathbb{I}_n$.
\end{proof}

\begin{rem} We observe the following:
\begin{itemize}
	\item[(i)] If $C_H=0$, namely $H$ is independent of $x$, then the rate of convergence depends only on the initial condition and not on the H\"older regularity properties of $u$ at positive times. 

\item[(ii)] $C_1=0$ if the initial condition is constant: in this case the rate does not depend on the initial datum whereas the case $\alpha=0$, that is $u(t)$ bounded, is also permitted.

\item[(iii)] We allow for the case $\alpha >\eta$, namely when $u(t)$ belongs to a higher H\"older scale than $u_0$ as a conseguence of a regularizing effect of the Hamiltonian, independent of the viscosity. However, the decay of the function $t\mapsto [u(t)]_\alpha$ as $t\to 0^+$ has enough integrability. 

\item[(iv)] The constant of the main estimate depends on the dimension $n$, the bound $\Lambda$ on $F$ and on small time $t$ as a power with the same exponent $\tfrac{P}{P+1}$ of the viscosity parameter $\epsilon$. 

\item[(v)] Note that $u$ and $u_{(\epsilon)}$ are not assumed to be bounded and the convergence rate does not depend on their sup-norms.  
\item[(vi)] The case $u_{(\epsilon)}(0) \neq u_0$ can be treated in the same way: in such case, $||u_0- u_{(\epsilon)}(0)||_\infty$ is added to the reminder in \eqref{eq:main}.
\item[(vii)] The proof above also yields that, if $u$ and $v$ are uniformly continuous, $u$ is a viscosity subsolution to \eqref{HJintro} and $v$ a viscosity subsolution to \eqref{HJintro}, then there exists a modulus of continuity $\omega$ such that for any  $x\in \R^n$ and  $t\in (0,T]$
\[
u(x,t)- u_{(\epsilon)}(x,t) \leq \omega(\epsilon), 
\qquad u_{(\epsilon)}(x,t) - v(x,t)\leq \omega(\epsilon).
\]
Thus $u(x,t) -v(x,t) \leq 2 \omega(\epsilon)$ which, sending $\epsilon\to 0^+$, gives $u(x,t) \leq v(x,t)$: this is a proof of the comparison principle for uniformly continuous viscosity solutions to \eqref{HJintro}, by assuming that the comparison principle holds for a possibly regularized version of the equation.
\end{itemize}
\end{rem}

As a corollary of the main result above, we deduce a rate of convergence towards the initial datum for certain fully nonlinear evolution equations.
\begin{equation}\label{fully}
\begin{cases}
\partial_t u_{(\varepsilon)}(x,t)-\varepsilon F(x,t,D^2u_{(\varepsilon)}(x,t))=0&\text{ in }\R^n\times(0,T),\\
u_{(\varepsilon)}(x,0)=u_0(x)&\text{ in }\R^n.
\end{cases}
\end{equation}
This extends to the fully nonlinear setting \cite[Lemma 1.14]{BertozziMajda}.
\begin{cor}\label{mainfully}
If $u_{(\varepsilon)}$ solves \eqref{fully} and $\|Du_0\|_\infty< \infty$, 
then
\[
\|u_{(\eps)}(\cdot, t)-u_0(\cdot)\|_{L^\infty(\R^n)}\leq 4\|Du_0\|\sqrt{\varepsilon t}+C_F t\eps.
\]
\end{cor}

We conclude by stating the result for the vanishing viscosity approximation of the stationary equation
\begin{equation}\label{stat}
\rho  u+H(x,Du,D^2u)=0\text{ in }\R^n,\ \rho>0.
\end{equation}
\begin{thm}\label{mainstat}
Under the assumption of Theorem \ref{main} (including $\alpha=0$) with $H$ independent of $t$, we have
\[
\|u_{(\varepsilon)}-u\|_{L^\infty(\R^n)}\leq \frac1\rho\left[C_3^{\frac{1}{Q+1}}\left(\frac{Q+1}{Q^{\frac{Q}{Q+1}}}\right)(\Lambda n\varepsilon)^{\frac{Q}{Q+1}}+\eps C_F\right],
\]
where
\begin{align*}
Q&:=\frac{\beta+\gamma(\alpha-1)}{2-\alpha};\\
\frac{Q}{Q+1}&=\frac{\beta+\gamma(\alpha-1)}{\beta+\gamma(\alpha-1)+2-\alpha};\\
C_3&:=C_H(2[u]_\alpha)^{\frac{\beta}{2-\alpha}}\left(1+(2[u]_\alpha)^{\frac{\gamma}{2-\alpha}}\right).
\end{align*}
\end{thm}
\begin{proof}
The proof of this result follows the main steps of Theorem \ref{main}, taking into account the absence of the initial condition.
\end{proof}
\begin{rem}
In this case we provide a new rate of order $\mathcal{O}(\eps^{\frac{\beta}{\beta+2}})$ when $u$ is just bounded ($\alpha=0$). It is easy to see that when $C_H=C_F=0$, i.e. $H=H(p,M)$ and $F=F(M)$ are independent of $x$ satisfies $F(0)=0$, the right-hand side of the estimate vanishes; indeed in this case \\ $u_{(\eps)}=u=-\frac{H(0,0)}{\rho}$ is constant.
\end{rem}

%\bibliography{stability}

\begin{thebibliography}{QSTY24}

\bibitem[BBDL98]{BBDL}
Martino Bardi, Sandra Bottacin, and Francesco Da~Lio.
\newblock Soluzioni di viscosit\`a di equazioni nonlineari ellittiche degeneri.
\newblock 1998.

\bibitem[BCD97]{BCD}
Martino Bardi and Italo Capuzzo-Dolcetta.
\newblock {\em Optimal control and viscosity solutions of
  {H}amilton-{J}acobi-{B}ellman equations}.
\newblock Systems \& Control: Foundations \& Applications. Birkh\"{a}user
  Boston, 1997.

\bibitem[Bre20]{Brenier}
Yann~Brenier.
\newblock Examples of hidden convexity in nonlinear {PDEs}, {L}ecture {N}otes
  at {ETH} {Z}urich.
\newblock hal-02928398, 2020.

\bibitem[Cal18]{Calder}
Jeff~Calder.
\newblock {L}ecture notes on viscosity solutions.
\newblock available at
  \url{https://www-users.cse.umn.edu/~jwcalder/viscosity_solutions.pdf}, 2018.

\bibitem[CC95]{CC}
Luis~A. Caffarelli and Xavier Cabr\'{e}.
\newblock {\em Fully nonlinear elliptic equations}, volume~43 of {\em American
  Mathematical Society Colloquium Publications}.
\newblock American Mathematical Society, Providence, RI, 1995.

\bibitem[CCM11]{CCM}
Fabio Camilli, Annalisa Cesaroni, and Claudio Marchi.
\newblock Homogenization and vanishing viscosity in fully nonlinear elliptic equations: rate of convergence estimates.
\newblock {\em Adv. Nonlinear Stud.}, 11(2):405--428, 2011.

\bibitem[CD25]{ChaintronDaudin}
Louis-Pierre Chaintron and Samuel~Daudin.
\newblock Optimal rate of convergence in the vanishing viscosity for uniformly convex
  {H}amilton-{J}acobi equations.
\newblock  Preprint, {arXiv}:2506.13255 [math.{AP}] (2025).



\bibitem[CG25]{CG25}
Marco Cirant and Alessandro Goffi.
\newblock Convergence rates for the vanishing viscosity approximation of
  {Hamilton}-{Jacobi} equations: the convex case.
\newblock Preprint, {arXiv}:2502.15495 [math.{AP}] (2025), to appear in Indiana
  Univ. Math. J., 2025.

\bibitem[CGM23]{CGM}
Fabio Camilli, Alessandro Goffi, and Cristian Mendico.
\newblock Quantitative and qualitative properties for {Hamilton}-{J}acobi
  {PDE}s via the nonlinear adjoint method, to appear in Ann. Sc. Norm. Super.
  Pisa, Cl. Sci., 2024.

\bibitem[CIL92]{CIL}
Michael~G. Crandall, Hitoshi Ishii, and Pierre-Louis Lions.
\newblock User's guide to viscosity solutions of second order partial
  differential equations.
\newblock {\em Bull. Am. Math. Soc., New Ser.}, 27(1):1--67, 1992.

\bibitem[CL84]{CL84}
Michael G. Crandall and Pierre-Louis Lions.
\newblock Two approximations of solutions of {H}amilton-{J}acobi equations.
\newblock {\em Math. Comp.}, 43(167):1--19, 1984.

\bibitem[Eva10]{EvansArma}
Lawrence~C. Evans.
\newblock Adjoint and compensated compactness methods for {H}amilton-{J}acobi
  {PDE}.
\newblock {\em Arch. Ration. Mech. Anal.}, 197(3):1053--1088, 2010.

\bibitem[Fuj18]{Fujita}
Yasuhiro Fujita.
\newblock Lower estimates of {$L^\infty$}-norm of gradients for {C}auchy
  problems.
\newblock {\em J. Math. Anal. Appl.}, 458(2):910--924, 2018.

\bibitem[Lio82]{L82book}
Pierre-Louis Lions.
\newblock {\em Generalized solutions of {H}amilton-{J}acobi equations},
  volume~69 of {\em Research Notes in Mathematics}.
\newblock Pitman (Advanced Publishing Program), Boston, Mass.-London, 1982.

\bibitem[LL86]{LasryLions}
Jean-Michel Lasry and Pierre-Louis Lions.
\newblock A remark on regularization in {H}ilbert spaces.
\newblock {\em Israel J. Math.}, 55(3):257--266, 1986.

\bibitem[LS24]{LionsSeeger}
Pierre-Louis Lions and Benjamin Seeger.
\newblock Transport equations and flows with one-sided {L}ipschitz velocity
  fields.
\newblock {\em Arch. Ration. Mech. Anal.}, 248(5):Paper No. 86, 61, 2024.

\bibitem[LT01]{LinTadmor}
Chi-Tien Lin and Eitan Tadmor.
\newblock {$L^1$}-stability and error estimates for approximate
  {H}amilton-{J}acobi solutions.
\newblock {\em Numer. Math.}, 87(4):701--735, 2001.

\bibitem[MB02]{BertozziMajda}
Andrew~J. Majda and Andrea~L. Bertozzi.
\newblock {\em Vorticity and incompressible flow}.
\newblock Camb. Texts Appl. Math. Cambridge: Cambridge University Press, 2002.

\bibitem[QSTY24]{Sprekeleretal}
Jianliang Qian, Timo Sprekeler, Hung~V. Tran, and Yifeng Yu.
\newblock Optimal rate of convergence in periodic homogenization of viscous
  {Hamilton}-{Jacobi} equations.
\newblock Mult. Model. Simul. 22, No. 4, 1558--1584, 2024.

\bibitem[Sou85]{Souganidis}
Panagiotis~E. Souganidis.
\newblock Existence of viscosity solutions of {H}amilton-{J}acobi equations.
\newblock {\em J. Differential Equations}, 56(3):345--390, 1985.

\bibitem[Tra21]{TranBook}
Hung~V. Tran.
\newblock {\em Hamilton-Jacobi equations: {T}heory and {A}pplications}.
\newblock AMS Graduate Studies in Mathematics. American Mathematical Society,
  2021.

\end{thebibliography}
%\bibliographystyle{alpha}

\end{document}